\documentclass[10pt]{article}
\usepackage{amsmath,amssymb,amsfonts,amsthm}
\usepackage{epsf,color}
\usepackage{fancyhdr}
\newtheorem{theorem}{Theorem}
\newtheorem{lemma}{Lemma}
\newtheorem{prop}{Proposition}
\newtheorem{corollary}{Corollary}

\newcommand{\lb}{\left (}
\newcommand{\rb}{\right )}
\newcommand{\lbb}{\left [}
\newcommand{\rbb}{\right ]}
\newcommand{\labs}{\left |}
\newcommand{\rabs}{\right |}
\newcommand{\lbrb}[1]{\lb #1 \rb}
\newcommand{\lbbrbb}[1]{\lbb#1\rbb}

\newcommand{\lbcurly}{\left\{}
\newcommand{\rbcurly}{\right\}}

\newcommand{\abs}[1]{\labs#1\rabs}

\newcommand{\curly}[1]{\lbcurly#1\rbcurly}

\newcommand{\bo}[1]{\mathrm{O}\!\lbrb{#1}}
\newcommand{\so}[1]{\mathrm{o}\!\lbrb{#1}}
\newcommand{\soo}[1]{\mathrm{o}\lbrb{#1}}

\newcommand{\tightoverset}[2]{%
    \mathop{#2}\limits^{\vbox to 0.28ex{\kern -0.2ex\hbox{$#1$}\vss}}}

\newcommand{\Pbb}[1]{\Pb\lb #1\rb}
\newcommand{\Ebb}[1]{\Eb\lbb #1\rbb}

\newcommand{\limi}[1]{\lim\limits_{#1\to \infty}}

\newcommand{\Eb}{\mathbb{E}}

\newcommand{\Rb}{\mathbb{R}}

\newcommand{\Pb}{\mathbb{P}}

\newcommand{\ind}[1]{\mathbb{I}_{\{#1\}}}

\renewcommand{\Re}{\mathtt{Re}}

\usepackage{hyperref}

\vspace{.5cm}

\title{On the Maximal Multiplicity of Block Sizes in a Random Set Partition}

\bigskip

\author{ {\sc L. Mutafchiev}\thanks{American University in
Bulgaria, 1 "Georgi Izmirliev" Square,
 Blagoevgrad 2700, Bulgaria \  E-mail: Ljuben@aubg.edu}\,\,\thanks{Institute of Mathematics and Informatics, Bulgarian Academy of Sciences,
  "Akad. Georgi Bonchev" bl. 8, Sofia 1113, Bulgaria}  \: {\sc and}  {\sc M.~Savov}\thanks{Institute of
  Mathematics and Informatics, Bulgarian Academy of Sciences,  "Akad. Georgi Bonchev" bl. 8, Sofia 1113, Bulgaria \  E-mail: mladensavov@math.bas.bg}}
\date{}
\begin{document}
\maketitle

\begin{abstract}
We study the asymptotic behavior of the maximal multiplicity
$M_n=M_n(\sigma)$ of the block sizes in a set partition $\sigma$
of $[n]=\{1,2,....,n\}$, assuming that
$\sigma$ is chosen uniformly at
random from the set of all such partitions. It is known that, for
large $n$, the blocks of a random set partition are typically of
size $W=W(n)$, with $We^W=n$. We show that, over subsequences
$\{n_k\}_{k\ge 1}$ of the sequence of the natural numbers,
$M_{n_k}$, appropriately normalized, converges weakly, as
$k\to\infty$, to $\max{\{Z_1,Z_2-u\}}$, where $Z_1$ and $Z_2$ are
independent copies of a standard normal random variable. The
subsequences  $\{n_k\}_{k\ge 1}$, where the weak convergence is
observed, and the quantity $u$ depend on the fractional part $f_n$
of the function $W(n)$. In particular, we establish that
$\lim_{k\to\infty}\lbrb{\frac{1}{2\pi}}^\frac14
\min{\{f_{n_k},1-f_{n_k}\}}\sqrt{n_k}/\log^{7/4}{n_k}=u\in
[0,\infty)\cup\{\infty\}$. The behavior of the largest
multiplicity $M_n$ is in a striking contrast to the similar
statistic of integer partitions of $n$. A heuristic explanation of
this phenomenon is also given.
\end{abstract}

\vspace {.5cm}

\vspace{.5cm}


\vspace{.2cm}

\section{Introduction}

A partition $\sigma$ of the set $[n]:=\{1,2,...,n\}$ is a
representation of $[n]$ as a union of disjoint, non-empty subsets
called blocks. A block has size $j,j=1,2,...,n$, if it has
cardinality $j$. A partition $\sigma$ of $[n]$ obviously defines
the representation:
\begin{equation}\label{decompo}
n=\sum_{j=1}^n j\mu(j),
\end{equation}
where $\mu(j)\ge 0$ denotes the multiplicity (frequency) of blocks
of size $j$. The total number of set partitions of $[n]$, $B_n$,
is called $n$-th Bell number. It is well known that
\begin{equation}\label{gfbell}
\sum_{n=0}^\infty \frac{B_n}{n!} x^n=e^{e^x -1},
 \end{equation}
 where $x$ is a formal (complex) variable and $B_0:=1$ (regarding that $\emptyset$
 has exactly one partition -\,the empty partition). Moreover, (\ref{gfbell}) implies the following
 formula for Bell numbers:
$$
 B_n=e^{-1}\sum_{k=0}^\infty \frac{k^n}{k!}.
$$
 An asymptotic representation of the numbers $B_n$, as
$n\to\infty$, was found first by Moser and Wyman \cite{MW55}. We
have
\begin{equation*}
\frac{B_n}{n!}\sim \frac{e^{e^W-1}}{W^n\sqrt{2\pi W(W+1)e^W}},
\quad n\to\infty,
\end{equation*}
where $W=W(n)$ is the unique positive root of the equation
\begin{equation}\label{eq}
We^W=n,
\end{equation}
from which one can get
\begin{equation}\label{w}
W(n)=\log{n}-\log{\log{n}} +\bo{\frac{\log{\log{n}}}{\log{n}}}, \quad n\to\infty.
\end{equation}
For more details, we refer the reader, e.g., to \cite[Chapter
6]{D58} and \cite[Chapter VIII.4]{FS09}.

Further, we  introduce the uniform probability measure $P$ on the
set of all partitions of $[n]$ assuming that the probability
$1/B_n$ is assigned to each $n$-set partition $\sigma$. In this
way, each numerical characteristic of $\sigma$ becomes a random
variable (a statistic in the sense of the random generation of set
partitions of $[n]$). In the following, we will be interested in
the asymptotic behavior of the maximal multiplicity of the block
sizes defined by
 \begin{equation}\label{rv}
 M_n=M_n(\sigma):=\max_{1\le j\le n} {\mu(j)},
\end{equation}
where the multiplicities $\mu(j)$ were defined by (\ref{decompo}).
From one side, our study is motivated by the asymptotic results
related to limiting distributions of several set partition
statistics. On the other hand, we are interested in a comparison
between the typical behavior of $M_n$ for large $n$ and that one
of the similar statistic for the integer partitions of $n$
obtained in \cite{M05}. In the brief survey that we present below
we summarize several important results on the asymptotic
enumeration and the probabilistic study of set partitions.

Harper \cite{H67} was apparently the first who has studied set
partitions using a probabilistic approach. As a matter of fact, he
found an appropriate normalization for the total number of blocks
$Y_n$ in a random partition of $[n]$ and showed that
$(\log{n})Y_n/\sqrt{n}-\sqrt{n}$ converges weakly, as
$n\to\infty$, to a standard normal random variable. Sachkov
\cite{S74}, \cite{S78} obtained a multidimensional local limit
theorem for the multiplicities of blocks of sizes $1,2,...,k$ ($k$
being fixed) and found that the largest block size is asymptotic
to $eW$ (see (\ref{eq}) and (\ref{w})) plus a doubly exponentially
distributed random variable. De Laurentis and Pittel \cite{DP83}
proved that most of the blocks are likely to have sizes close to
$W$ ($\sim\log{n}$), see (\ref{w}), and the process counting those
typical blocks converges - in terms of finite dimensional
distributions - to the Brownian Bridge process. Another important
statistic of random set partitions is the total number
$\tilde{Y}_n$ of distinct block sizes. Odlyzko and Richmond
\cite{OR85} showed that $\tilde{Y}_n$ is asymptotic to $eW\sim
e\log{n}$ both in probability and in the mean, while Arratia and
Tavar\'{e} \cite{AT94} extended this result to the convergence of
$\tilde{Y}_n$ in $r$-th mean for every $r\ge 1$. In addition Goh
and Schmutz \cite{GS92} determined the limiting behavior of the
probability that the random set partition has at least one block
of every size less than the largest block size (a gap-free
partition).

A unified approach to a broad class of random combinatorial
structures problems was proposed by Arratia and Tavar\'{e}
\cite{AT94}. It covers also the case of random set partitions. For
more details, we also refer the reader to the subsequent book
\cite{ABT03} (see, in particular, its Chapters 3 and 7). Their
approach is based on a possibility of interpreting the
multiplicities of the blocks in a random partition of $[n]$ as a
specially constructed sequence of independent and Poisson
distributed random variables $\{V_j\}_{j\ge 1}$, whose parameters
are $W^j/j!$, conditioned on the event $\{\sum_{j\ge 1} jV_j=n\}$.
Such conditioning has been also used to study many other
combinatorial structures (e.g., permutations, mappings of a finite
set into itself, integer partitions, etc.; see \cite{ABT03}).
Among many other results, Arratia and Tavar\'{e} \cite{AT94}
obtained estimates for the total variation distance between proper
segments of block size multiplicities and such segments from the
sequence $\{V_j\}_{j\ge 1}$. In a subsequent study Pittel
\cite{P97} confirmed some conjectures of Arratia and Tavar\'{e}
\cite{AT94}, re-derived Sachkov's limiting distribution of the
maximal block size and obtained a functional limit theorem for a
continuous time process composed from block sizes of order
 $$
k(t) =\left\{\begin{array}{ll} \lfloor W+\Phi^{-1}(t)W^{1/2}\rfloor & \qquad  \mbox {if}\qquad t\in (0,1/2], \\
 \lceil W+\Phi^{-1}(t)W^{1/2}\rceil & \qquad \mbox {if}\qquad t\in (1/2,1].
 \end{array}\right.
 $$
 ($\Phi(t)$ is the cumulative distribution function of the
 standard normal random variable). A weak convergence to the
 Brownian bridge process was established.



 Our aim in this paper is to determine asymptotically, as
 $n\to\infty$, the distribution of the maximal multiplicity of
 block sizes $M_n$, defined by (\ref{rv}). To state our main
 result, we need to introduce some notations. With $W=W(n)$ given
 by (\ref{eq}) and (\ref{w}), we set
 \begin{equation}\label{d}
 d_n:=\lfloor W(n)\rfloor,
 \end{equation}
 \begin{equation}\label{f}
 f_n:=W(n)-d_n\in (0,1),
 \end{equation}
 \begin{equation}\label{r}
 R_n:=\frac{W^{d_n}}{d_n!},
 \end{equation}
 \begin{equation}\label{var}
 \vartheta_n:=\min{\{f_n,1-f_n\}}.
 \end{equation}
 Furthermore, let $Z_1$ and $Z_2$ denote two independent copies of
the standard normal random variable.

 We organize the paper as follows. Section \ref{sec:aux} contains some
 auxiliary facts related to the generating functions that we need
 further. In Section \ref{sec:proof} we prove the following limit theorem for
 $M_n$.

 \begin{theorem}\label{thm:main}
    We have the following scenarios.

 (i) If $\vartheta_{n_k}
 =\so{\frac{\log^{7/4}{n_k}}{\sqrt{n_k}}}$ over a
 subsequence $\{n_k\}_{k\ge 1}$, then
 $\frac{M_{n_k}-R_{n_k}}{\sqrt{R_{n_k}}}$ converges weakly, as
 $k\to\infty$, to $\max{\{Z_1,Z_2\}}$.

 (ii) If $\frac{\log^{7/4}{n_k}}{\sqrt{n_k}}=\so{\vartheta_{n_k}}$
over a
 subsequence $\{n_k\}_{k\ge 1}$, then
  $\frac{M_{n_k}-R_{n_k}}{\sqrt{R_{n_k}}}$ converges weakly,
  as $k\to\infty$, to $Z_1$.

 (iii) If
 $\lim_{k\to\infty}\lbrb{\frac{1}{2\pi}}^\frac14\vartheta_{n_k}\frac{\sqrt{n_k}}{\log^{7/4}{n_k}}=u\in\lbrb{0,\infty}$
over a
 subsequence $\{n_k\}_{k\ge 1}$, then $\frac{M_{n_k}-R_{n_k}}{\sqrt{R_{n_k}}}$ converges weakly,
 as $k\to\infty$, to $\max{\{Z_1,Z_2-u\}}$.

 In addition, all these three cases are possible.
 \end{theorem}

 We believe that the following comments on the comparison between
 this result and the corresponding one for random integer
 partitions would be helpful.

 It is shown in \cite{M05} that the maximal part size multiplicity
 in a random integer partition of $n$, when appropriately
 normalized, converges weakly, as $n\to\infty$, to the maximum of an infinite
 sequence of independent and exponentially distributed random
 variables. Theorem 1 establishes completely different limiting
 behavior for $M_n$. The next heuristic explains this phenomenon.

 We recall that the conditioning relation, given briefly above and
 stated in detail in \cite[Chapter 2]{ABT03}
 holds in both cases of random set partitions of $[n]$ and random integer
 partitions of $n$. So, in both cases the joint distribution of
 the multiplicities is transferred to the joint
 distribution of independent random variables $\{V_j\}_{j\ge1}$,
 conditioned upon $\{\sum_{j\ge 1} jV_j=n\}$. For set partitions,
 $V_j$ are Poisson distributed variables whose means are equal to
 $\lambda_j=W^j/j!$ and tend to infinity as $n\to\infty$. For
 integer partitions, $V_j$ are geometrically distributed with parameters
 $e^{-j\pi/\sqrt{6n}}$. The last fact was first established by
 Fristedt \cite{F93}. These relationships explain the major
 difference in the multidimensional limit laws of the
 multiplicities of set partitions and integer partitions. In fact,
 Poisson distributions with growing parameters are well
 approximated by the normal law, which is shown in the set partition case by Sachkov
 \cite[Chapter IV, Theorem 4.2]{S78}, while the geometric
 distributions are asymptotically close to the exponential law
 confirmed for integer partitions by Theorems 2.1 and 2.2 in \cite{F93}. Hence, the
 limiting distribution of the maximum multiplicity in set
 partitions is based on the normal distribution, while for integer
 partitions we observe the effect of the exponential distribution.

 Furthermore, Theorem 1 shows that the limiting behavior of $M_n$
 depends on the random variable $V_{d_n}$ and its two
 closest neighbors in the sequence $\{V_j\}_{j\ge 1}$ at most. This
 property is based on the fact that, for different $j$'s, the variables $V_j$ converge to the standard
 normal random variable under different normalization depending on
 their parameters $\lambda_j$. One can easily observe the
 fast growth of the ratios $\lambda_{j+1}/\lambda_j$ for $j\le
 d_n-1$ and their fast decrease for $j\ge d_n+1$, This implies
 that the normalization of $M_n$ must only depend on the maximum
 Poisson parameter and eventually on its neighbors
   in the sequence $\{\lambda_j\}_{j\ge 1}$, i.e., on  $\lambda_{d_n}=R_n$ given by (\ref{r})
   and possibly on $\lambda_{d_n\pm 1}$. The
 participation of the fractional part $f_n$ of $W$ in Theorem 1 is due to the
 application of the Stirling's formula to the parameter
 $\lambda_{d_n}$. In the case of integer partitions we observe a
 completely different phenomenon: the multiplicity of part $j$
 multiplied by its size $j$ tends weakly to one and the same
 exponentially distributed random variable with the same normalization
 (equal to $\sqrt{6n}/\pi$) for all $j\ge 1$; see \cite[Theorems 2.1 and 2.2]{F93}. Hence all parts of a random
 integer partition contribute to the limiting behavior of the
 maximal part size in a likely manner that depends only on the
 part size $j$ itself.

 In our proof we use the saddle point method. We encounter several
 technical difficulties in its application since the underlying integrand is of the form of a product $f(x)g(x)$. This case was
 discussed in detail by Odlyzko in his survey
 \cite[p.1183] {O95}. Our integrand involves the set partition generating
 function, i.e., we have $f(x)=e^{e^x-1}$. The second factor depends
 on an extra parameter $m=m(n)$, i.e., $g(x)=g(x;m)$. The latter
 one
 remains bounded when $x$ is near to the saddle point $x=W$. We
 establish an
 asymptotic of the general type as it is given in
 \cite[formula (12.43), p. 1183]{O95}. Probabilistic and asymptotic analysis of $g(x;m)|_{x=W}$
 in turn yields the results of Theorem 1.

\section{Preliminaries}\label{sec:aux}

We start with a generating function identity for the cumulative
distribution function of the maximal block size multiplicity $M_n$
defined by (\ref{rv}).

\begin{lemma}\label{lem:gf} For any formal (complex) variable $x$ and any $m\ge 1$, we have
$$
\sum_{n=0}^\infty\frac{B_n}{n!}P(M_n\le m) x^n =e^{e^x-1} F_m(x),
$$
where, by convention, $B_0=P(M_0\le 0)=1$ and
\begin{eqnarray}\label{fm}
& & F_m(x) =\prod_{j=1}^\infty\left(1-e^{-x^j/j!}\sum_{k>m}^\infty
\left(\frac{x^{j}}{j!}\right)^k\frac{1}{k!}\right) \nonumber \\
& & =\prod_{j=1}^\infty\left(e^{-x^j/j!}\sum_{k=0}^m
\left(\frac{x^{j}}{j!}\right)^k \frac{1}{k!}\right).
\end{eqnarray}
\end{lemma}

\begin{proof}
 The proof is based on a general identity established in
 \cite[Chapter III, formula (0.14)]{S78}. We set therein
 $\Lambda_j=\{0,1,...,m\}$ for all $j\ge 1$ and obtain
 \begin{eqnarray}
 & & \sum_{n=0}^\infty\frac{B_n}{n!}P(M_n\le m) x^n
 \nonumber \\
 & & =\prod_{j=1}^\infty\left(\sum_{k=0}^m
\left(\frac{x^{j}}{j!}\right)^k \frac{1}{k!}\right)
=\prod_{j=1}^\infty\left(e^{x^j/j!} -\sum_{k>m}^\infty
\left(\frac{x^{j}}{j!}\right)^k\frac{1}{k!}\right) \nonumber \\
& & =\prod_{j=1}^\infty
e^{x^j/j!}\left(1-e^{-x^j/j!}\sum_{k>m}^\infty
\left(\frac{x^{j}}{j!}\right)^k\frac{1}{k!}\right) \nonumber \\
& & =e^{e^x-1}
\prod_{j=1}^\infty\left(1-e^{-x^j/j!}\sum_{k>m}^\infty
\left(\frac{x^{j}}{j!}\right)^k\frac{1}{k!}\right), \nonumber
\end{eqnarray}
which completes the proof of the lemma.
\end{proof}

Further, we apply the Cauchy coefficient formula to the generating
function identity of Lemma \ref{lem:gf} on the circle $x=W
e^{i\theta}, -\pi<\theta\le\pi$ with $W=W(n)$ determined by
(\ref{eq}) and (\ref{w}). Thus we obtain
 \begin{equation}\label{eq:integral}
    \begin{split}
    & \frac{B_n}{n!} P(M_n\leq m)=\frac{W^{-n}}{2\pi}\int_{-\pi}^{\pi} e^{e^{We^{i\theta}}-1}F_{m}\lbrb{We^{i\theta}}e^{-i\theta
    n}d\theta:=J_n.
    \end{split}
    \end{equation}

    In the asymptotic analysis of the integral $J_n$ we will use
    characteristic functions. Hence we need a more convenient representation of
    the second factor of the integrand in (\ref{eq:integral}) (see also
    (\ref{fm})). We recall that in the Introduction we defined the sequence of independent and Poisson
    distributed random variables $\{V_j\}_{j\ge 1}$ with parameters
    \begin{equation}\label{eq:r}
\begin{split}
&\lambda_j:=\frac{W^j}{j!}.
\end{split}
\end{equation}
    Further on, we denote by $\mathbb{P}$ the probability measure
    on the probability space, where the sequence $\{V_j\}_{j\ge 1}$ is defined.
     The expectation
    with respect to the probability measure $\mathbb{P}$ is denoted by $\mathbb{E}$. We also denote by
    $\ind{A}$ the indicator of an event $A$ from the same probability space.

\begin{lemma}\label{lem:F}
        We have that
        \begin{equation}\label{eq:repF}
        \begin{split}
        F_m(We^{i\theta})e^{-i\theta n}
        &=e^{e^W-e^{We^{i\theta}}}\prod_{j=1}^{\infty}\Ebb{e^{i\theta j (V_j-\lambda_j)} \ind{V_j\leq m}}\\
        &=e^{e^W-e^{We^{i\theta}}}\Ebb{e^{i\theta \sum_{j=1}^{\infty}j(V_j-\lambda_j)}\ind{\bigcap_{j=1}^\infty \curly{V_j\leq
        m}}}.
        \end{split}
        \end{equation}
            \end{lemma}

    \begin{proof}
    The proof is relatively straightforward.  It relies on the independence of $\{V_{j}\}_{j\geq 1}$, the fact that
    \begin{equation}\label{eq:trun}
    \begin{split}
    &\Ebb{e^{i\theta j V_j}\ind{V_j\leq m}}=e^{-\frac{W^j}{j!}}\sum_{k=0}^me^{i\theta j k}\lbrb{\frac{W^j}{j!}}^k\frac{1}{k!},\,\,\theta\in\Rb,
    \end{split}
    \end{equation}
    and the obvious identity
    \begin{equation}\label{eq:a}
    \begin{split}
    &\sum_{j=1}^\infty j\lambda_j=\sum_{j=1}^\infty j\frac{W^j}{j!}=We^W=n.
    \end{split}
    \end{equation}
    Indeed, we have from \eqref{fm} and \eqref{eq:trun} that
    \begin{equation*}
    \begin{split}
    F_m(We^{i\theta})e^{-i\theta n}&=e^{-i\theta n}\prod_{j=1}^{\infty}\lbrb{e^{-e^{i\theta j}\frac{W^j}{j!}}\sum_{k=0}^{m}e^{i\theta j k}\lbrb{\frac{W^j}{j!}}^k\frac{1}{k!}}\\
    &=e^{-i\theta n}\prod_{j=1}^{\infty}e^{\frac{W^j}{j!}-e^{i\theta j}\frac{W^j}{j!}}\Ebb{e^{i\theta j V_j} \ind{V_j\leq m}}\\
    &=e^{e^W-e^{We^{i\theta}}}\prod_{j=1}^{\infty}\Ebb{e^{i\theta j (V_j-\lambda_j)} \ind{V_j\leq m}}\\
    &=e^{e^W-e^{We^{i\theta}}}\Ebb{e^{i\theta \sum_{j=1}^{\infty}j(V_j-\lambda_j)}\ind{\bigcap_{j=1}^\infty \curly{V_j\leq m}}},
    \end{split}
    \end{equation*}
    which confirms \eqref{eq:repF}. We only note that, for fixed
    $W$,
    $\sum_{j=1}^{\infty}j(V_j-\lambda_j)<\infty$ almost surely since
    \begin{equation*}
    \begin{split}
    &   Var\lbrb{\sum_{j=1}^{\infty}j(V_j-\lambda_j)}=\sum_{j=1}^\infty j^2\lambda_j=\sum_{j=1}^\infty j^2\frac{W^j}{j!}<\infty.
    \end{split}
    \end{equation*}
\end{proof}

We conclude this section with a decomposition of the integral in
(\ref{eq:integral}). We break the range of integration up as
follows. First, we set
\begin{equation}\label{eq:delta_n}
\delta_n:=\frac{n^{\frac{1}{7}}}{\sqrt{n\log(n)}}
\end{equation}
and
\begin{equation}\label{eq:gamma n}
\gamma_n:=\frac{1}{\log^{\frac{1}{5}}{n}}.
\end{equation}
Then, we write $(-\pi,\pi]=D_1\bigcup D_2\bigcup D_3$, where
$D_1=\{\theta:-\delta_n<\theta<\delta_n\},
D_2=\{\theta:\delta_n\le|\theta|<\gamma_n\},
D_3=\{\theta:\gamma_n\le|\theta|<\pi\}$ and
\begin{equation}\label{eq:decompo}
J_n =J_{n,1}+J_{n,2}+J_{n,3},
\end{equation}
where
\begin{equation}\label{eq:jnk}
J_{n,k}=\frac{W^{-n}}{2\pi}\int_{D_k}e^{e^{We^{i\theta}}-1}F_{m}\lbrb{We^{i\theta}}e^{-i\theta
    n}d\theta, \quad k=1,2,3.
\end{equation}
The asymptotic analysis of these three integrals will be given in
the next section.

\section{Proof of Theorem \ref{thm:main}}\label{sec:proof}
Before we state a key result for our investigations we introduce and recall some notation.
 For each $n\geq 1$, omitting where obvious the dependence on $n$, we consider a sequence of
 independently distributed random variables $\{V_j\}_{j\ge 1}$, where $V_j$ is Poisson of parameter $\lambda_j$, see \eqref{eq:r}.
Set
\begin{equation*}
\begin{split}
&\bar{V}_n:=\max_{j\geq 1}\curly{V_{j}}.
\end{split}
\end{equation*}
Recall also the definitions of
$d_n$ and $R_n$; see \eqref{d} and \eqref{r}, respectively, and
denote by $\Phi$ the cumulative distribution function of the
standard normal law. Then we have the following claim.
\begin{theorem}\label{cor:max}
    For any $c\in\Rb$, we have
    \begin{equation}\label{eq:maxIm3}
    \begin{split}
    &\abs{\Pbb{\bar{V}_n\leq R_n-c\sqrt{R_n}}-\prod_{j=d_n-1;d_n;d_n+1}\Pbb{V_{j}\leq R_n-c\sqrt{R_n}}}=\so{e^{-\frac{n}{\log^6(n)}}}\!,
    \end{split}
    \end{equation}
    and moreover,
    \begin{equation}\label{eq:maxIm4}
    \begin{split}
    &\prod_{j=d_n-1}^{d_n+1}\Pbb{V_{j}\leq R_n-c\sqrt{R_n}}\\
    &=\Phi(-c)\prod_{j=d_n-1;d_n+1}\Pbb{V_{j}\leq R_n-c\sqrt{R_n}}+\so{1}.
    \end{split}
    \end{equation}
\end{theorem}
The proof of Theorem \ref{cor:max} relies on several intermediate technical results.
\begin{lemma}\label{lem:r}
    We have
    \begin{equation}\label{eq:maxPar}
    \begin{split}
    &R_n=\max_{j\geq 1}\curly{\lambda_j}=\lambda_{d_n}>\max_{j\geq 1;j\neq d_n}\curly{\lambda_j}.
    \end{split}
    \end{equation}
    Moreover, for any $n$, the sequence $\curly{\lambda_j}_{j\geq 1}$ strictly increases for $j\leq d_n$ and decreases afterwards.
    Finally, as $n\to\infty$,
    \begin{equation}\label{eq:Rn}
    \begin{split}
    &R_n\sim \frac{1}{\sqrt{2\pi}}\frac{n}{\log^{\frac{3}{2}}(n)}.
    \end{split}
    \end{equation}

\end{lemma}
\begin{proof}
   We notice that
    $\frac{\lambda_{j+1}}{\lambda_j}=\frac{W}{j},$ which since $W=W(n)$ is transcendental,
     strictly exceeds $1$ if $j\leq d_n$ and is smaller than $1$ if $j\geq d_n+1$, see \eqref{d} for the definition of $d_n$.
      Thus, \eqref{eq:maxPar} and the claim that succeeds it follow. For the final claim we use the well-known Stirling asymptotic
    \begin{equation}\label{eq:Stir}
    \begin{split}
    &n!\sim \sqrt{2\pi}n^{n+\frac12}e^{-n}
    \end{split}
    \end{equation}
    to get
    \begin{equation*}
    \begin{split}
    &R_n=\frac{W^{d_n}}{d_n!}\sim \lbrb{\frac{W}{d_n}}^{d_n}\frac{e^{d_n}}{\sqrt{2\pi d_n}}=\lbrb{1+\frac{f_n}{d_n}}^{d_n}\frac{e^{d_n}}{\sqrt{2\pi d_n}}\\
    &\sim \frac{e^{W}}{\sqrt{2\pi W}}=\frac{n}{\sqrt{2\pi }W^{\frac32}}\sim \frac{1}{\sqrt{2\pi}}\frac{n}{\log^{\frac{3}{2}}(n)},
    \end{split}
    \end{equation*}
    where $f_n$ is defined in \eqref{f}, \eqref{eq} was employed for the last identity and \eqref{w} for the very last asymptotic relation.
    This concludes the proof.
\end{proof}
 Lemma \ref{lem:r} suggests that $V_{d_n}$ plays important role for the overall behavior of $\bar{V}_n$. From now on we use $\stackrel{d}{=}$ for identity in distribution between two random variables. Since it is a classical result that
\begin{equation}\label{eq:Norm}
\begin{split}
&\limi{n}\frac{V_{d_n}-R_n}{\sqrt{R_n}}\stackrel{d}{=}Z_1
\end{split}
\end{equation}
we proceed to investigate whether and
when the other Poisson random variables scale in the same fashion
with $R_n$. We have the following preliminary result.
\begin{lemma}\label{lem:scale}
    For any $c\in \Rb$, we have that
    \begin{equation}\label{eq:see1}
    \begin{split}
    &\frac{R_n-\lambda_{d_n+1}}{\sqrt{\lambda_{d_n+1}}}-c\frac{\sqrt{R_n}}{\sqrt{\lambda_{d_n+1}}}= \sqrt{R_n}\frac{1-f_n}{W(n)}\lbrb{1+\so{1}}-c+\so{1},\\
    &\frac{R_n-\lambda_{d_n-1}}{\sqrt{\lambda_{d_n-1}}}-c\frac{\sqrt{R_n}}{\sqrt{\lambda_{d_n-1}}}=
    \sqrt{R_n}\frac{f_n}{W(n)}\lbrb{1+\so{1}}-c+\so{1}.
    \end{split}
    \end{equation}
    Moreover, for any sequence $\curly{j_n}_{n\geq 1}$ such that $j_n\notin\curly{d_n-1,d_n,d_n+1}$, we have that
    \begin{equation}\label{eq:see3}
    \begin{split}
    &\limi{n}\lbrb{\frac{R_n-\lambda_{j_n}}{\sqrt{\lambda_{j_n}}}-c\frac{\sqrt{R_n}}{\sqrt{\lambda_{j_n}}}}=\infty.
    \end{split}
    \end{equation}
\end{lemma}
\begin{proof}
    Consider first $\lambda_{d_n+1}$. Note that
    \begin{equation}\label{eq:rec}
    \begin{split}
    \lambda_{d_n+1}=\frac{W^{d_n+1}}{(d_n+1)!}=R_n\frac{W}{d_n+1}
    \end{split}
    \end{equation}
    and since $d_n<W<d_n+1$ then it follows that
    \begin{equation}\label{eq:ration}
    \begin{split}
    &\limi{n}\frac{\lambda_{d_n+1}}{R_n}=1.
    \end{split}
    \end{equation}
    We then have
    \begin{equation*}
    \begin{split}
    &\frac{R_n-\lambda_{d_n+1}}{\sqrt{\lambda_{d_n+1}}}-c\frac{\sqrt{R_n}}{\sqrt{\lambda_{d_n+1}}}= R_n\frac{1-\frac{W}{d_n+1}}{\sqrt{R_n\frac{W}{d_n+1}}}-c+\so{1}\\
    &=\sqrt{R_n}\lbrb{1-\frac{W}{d_n+1}}\lbrb{1+\so{1}}-c+\so{1}\\
    &=\sqrt{R_n}\frac{1-f_n}{W}\lbrb{1+\so{1}}-c+\so{1}.
    \end{split}
    \end{equation*}
    The second claim of \eqref{eq:see1} follows in the same fashion. Let $j_n=d_{n}+2$. We demonstrate that \eqref{eq:see3} holds since
    \begin{equation*}
    \begin{split}
    &   \limi{n}\lbrb{\frac{R_n-\lambda_{d_n+2}}{\sqrt{\lambda_{d_n+2}}}-c\frac{\sqrt{R_n}}{\sqrt{\lambda_{d_n+2}}}}\\
    &= \limi{n}\lbrb{\frac{R_n}{\sqrt{\lambda_{d_n+2}}}\lbrb{1-\frac{W^2}{(d_n+1)(d_n+2)}}-c\frac{\sqrt{(d_n+1)(d_n+2)}}{\sqrt{W^2}}}\\
    &=   \limi{n}\lbrb{\frac{\sqrt{R_n}}{(d_n+1)(d_n+2)} \lbrb{(d_n+1)(d_n+2)-W^2}}-c\\
    &=      \limi{n}\lbrb{\frac{\sqrt{R_n}}{(d_n+1)(d_n+2)}\lbrb{3d_n+2-2d_nf_n-f^2_n}}-c\\
    &= \limi{n}\lbrb{\frac{\sqrt{R_n}}{W^2}\lbrb{(3-2f_n)W(n)+2-f^2_n}}-c=\infty,
    \end{split}
    \end{equation*}
    where we have used \eqref{eq:Rn} in the very last relationship.
    Let us next consider the case $j_n=d_{n}-2$. Then
    \begin{equation*}
    \begin{split}
    &   \limi{n}\lbrb{\frac{R_n-\lambda_{d_n-2}}{\sqrt{\lambda_{d_n-2}}}-c\frac{\sqrt{R_n}}{\sqrt{\lambda_{d_n-2}}}}\\
    &=\limi{n}\lbrb{\frac{R_n}{\sqrt{\lambda_{d_n-2}}}\lbrb{1-\frac{(d_n-1)(d_n-2)}{W^2}}-c\frac{W}{\sqrt{(d_n-1)(d_n-2)}}}\\
    &= \limi{n}\lbrb{\frac{\sqrt{R_n}}{W^2}\lbrb{W^2-(d_n-1)(d_n-2)}}-c \\
    &= \limi{n}\lbrb{     \frac{\sqrt{R_n}}{W^2}\lbrb{3d_n+2d_nf_n+f^2_n-2}}-c\\
    &= \limi{n}\lbrb{\frac{\sqrt{R_n}}{W^2}\lbrb{(3+2f_n)W-2+f^2_n}}-c=\infty,
    \end{split}
    \end{equation*}
    where we have used again \eqref{eq:Rn} in the very last relationship. Finally, consider $\curly{j_n}_{n\geq 1}$ such that $j_n\notin\curly{d_n-1,d_n,d_n+1}$ for all $n$ big enough.  
     Let first $j_n\leq d_n-2$. Then since $\lambda_{j_n}\leq \lambda_{d_n-2}$, see Lemma \ref{lem:r}, we have that
    \begin{equation}\label{eq:n}
    \begin{split}
    &       \frac{R_n-\lambda_{j_n}}{\sqrt{\lambda_{j_n}}}-c\frac{\sqrt{R_n}}{\sqrt{\lambda_{j_n}}}= \frac{1}{\sqrt{\lambda_{j_n}}}\lbrb{R_n-\lambda_{j_n}-c\sqrt{R_n}}\\
    &\geq \frac{1}{\sqrt{\lambda_{j_n}}}\lbrb{R_n-\lambda_{d_n-2}-c\sqrt{R_n}}\\
    &=\frac{\sqrt{\lambda_{d_n-2}}}{\sqrt{\lambda_{j_n}}}\lbrb{\frac{R_n-\lambda_{d_n-2}}{\sqrt{\lambda_{d_n-2}}}-c\frac{\sqrt{R_n}}{\sqrt{\lambda_{d_n-2}}}}\\
    &\geq \lbrb{\frac{R_n-\lambda_{d_n-2}}{\sqrt{\lambda_{d_n-2}}}-c\frac{\sqrt{R_n}}{\sqrt{\lambda_{d_n-2}}}}.
    \end{split}
    \end{equation}
   Then, we can get from the arguments above that
    \begin{equation*}
    \begin{split}
    & \limi{n}\frac{R_n-\lambda_{j_n}}{\sqrt{\lambda_{j_n}}}-c\frac{\sqrt{R_n}}{\sqrt{\lambda_{j_n}}}=\infty.
    \end{split}
    \end{equation*}
     The case $j_n\geq d_n+2$ is dealt with precisely as in \eqref{eq:n} since  $\lambda_{j_n}\leq \lambda_{d_n+2}$ from Lemma \ref{lem:r}. This concludes the claim.
\end{proof}

Lemma \ref{lem:scale} suggests that the distribution
of $\bar{V}_n$ may turn out to
depend at most on $V_{d_n-1},V_{d_n}, V_{d_n+1}$. We start
investigating this by recording the immediate formula
\begin{equation}\label{eq:maxIm}
\begin{split}
&\Pbb{\bar{V}_n\leq R_n-c\sqrt{R_n}}=\prod_{j=1}^{\infty}\Pbb{V_j\leq R_n-c\sqrt{R_n}},
\end{split}
\end{equation}
which is valid for any
$c\in \Rb$.
We proceed to study the infinite product without the terms $j\in\curly{d_n-1,d_n,d_n+1}$. We write, for any $c\in\Rb$,
\begin{equation*}
\begin{split}
&H_n:=R_n-c\sqrt{R_n}
\end{split}
\end{equation*}
and derive the following  property  of $H_n$.
\begin{lemma}\label{lem:H}
    For any $c\in\Rb$, there exists $n(c)$ such that, for $n\geq n(c)$,
    \begin{equation}\label{eq:prop}
    \begin{split}
    &\min_{j\neq d_n-1,d_n,d_n+1}\frac{H_n}{\lambda_{j}}=\min\curly{\frac{H_n}{\lambda_{d_n-2}};\frac{H_n}{\lambda_{d_n+2}}}>1.
    \end{split}
    \end{equation}
\end{lemma}
\begin{proof}
    Indeed, identity \eqref{eq:prop} is always true as $\lambda_{j}$ strictly increases for $j\leq d_n$ and strictly decreases afterwards, see Lemma \ref{lem:r}. For the inequality in \eqref{eq:prop} we note that
    \begin{equation*}
    \begin{split}
    &   \frac{H_n}{\lambda_{d_n+2}}=\frac{R_n-c\sqrt{R_n}}{R_n\frac{W^2}{(d_n+1)(d_n+2)}}=\frac{1-cR^{-1/2}_n}{\frac{W^2}{(d_n+1)(d_n+2)}}    >1\\
    &\iff   (d_n+1)(d_n+2)(1-cR^{-1/2}_n)>W^2\\
    &\iff (d^2_n+3d_n+2)(1-cR^{-1/2}_n)>d^2_n+2d_n+f^2_n\\
    &\iff d^2_n+3d_n+2-cR^{-1/2}_n(d^2_n+3d_n+2)>d^2_n+2d_nf_n+f^2_n\\
    &\iff 3d_n-2d_nf_n+2-cR^{-1/2}_n(d^2_n+3d_n+2)>f^2_n,
    \end{split}
    \end{equation*}
    where each successive implication involves a rearrangement of the initial inequality and an application of $W=f_n+d_n$ and of \eqref{eq:Rn}. The very last inequality is valid for all $n$ large enough from \eqref{eq:Rn} and $d_n=\lfloor W\rfloor\sim \log(n),$ see \eqref{w}, and $f_n\in\lbrb{0,1}$. Similarly one can check that
    \begin{equation*}
    \begin{split}
    &   \frac{H_n}{\lambda_{d_n-2}}=W^2\frac{R_n-c\sqrt{R_n}}{R_n(d_n-1)(d_n-2)}>1\\
    &\iff   W^2(1-cR^{-1/2}_n)>(d_n-1)(d_n-2)\\
    &\iff (d^2_n+2f_nd_n+f_n^2)(1-cR^{-1/2}_n)>d^2_n-3d_n+2\\
    &\iff 2f_nd_n+3d_n+f_n^2-cR^{-1/2}_n(d^2_n+2f_nd_n+f_n^2)>2.
    \end{split}
    \end{equation*}
   This concludes the proof of the claim.
\end{proof}
To simplify \eqref{eq:maxIm} we make some preliminary estimates regarding the terms in \eqref{eq:maxIm}.
\begin{prop}\label{prop:est}
    For any $c\in\Rb$ with $H_n=R_n-c\sqrt{R_n}$, we have
    \begin{equation}\label{eq:sum3}
    \begin{split}
    \limi{n}&\sum_{j=1;j\neq d_n-1,d_n,d_n+1}^\infty\Pbb{V_j>H_n}=0;\\
    &\sum_{j=1;j\neq d_n-1,d_n,d_n+1}^\infty\Pbb{V_j>H_n}=\so{e^{-\frac{n}{\log^6(n)}}}.
    \end{split}
    \end{equation}
\end{prop}
\begin{proof}
    We choose $n(c)$ so that \eqref{eq:prop} is valid for any $n\geq n(c)$.  Set $u_{j}=1-\frac{\lambda_j}{H_n}$ and from \eqref{eq:prop}, we arrive at
    \begin{equation*}
    \begin{split}
    \min_{j\neq d_n-1,d_n,d_n+1}u_{j}&=1-\max\curly{\frac{\lambda_{d_n-2}}{H_n};\frac{\lambda_{d_n+2}}{H_n}}\\
    &=\min\curly{1-\frac{\lambda_{d_n-2}}{H_n};1-\frac{\lambda_{d_n+2}}{H_n}}>0.
    \end{split}
    \end{equation*}
    We study the two terms in the last identity above. We have that
    \begin{equation}\label{eq:T1}
    \begin{split}
    &1-\frac{\lambda_{d_n+2}}{H_n}=\frac{1-\frac{c}{\sqrt{R_n}}-\frac{W^2}{(d_n+1)(d_n+2)}}{1-\frac{c}{\sqrt{R_n}}}\\
    &= 1-\frac{W^2}{(d_n+1)(d_n+2)}+\bo{\frac{\log^{\frac{3}{4}}(n)}{n^{\frac12}}}\\
    &= \frac{3d_n-2d_nf_n+2-f^2_n}{d^2_n}\lbrb{1+\bo{\frac{1}{d_n}}}+\bo{\frac{\log^{\frac{3}{4}}(n)}{n^{\frac12}}}\\
    &=\frac{3-2f_n+\so{1}}{d_n}\lbrb{1+\bo{\frac{1}{d_n}}}+\bo{\frac{\log^{\frac{3}{4}}(n)}{n^{\frac12}}}\\
    &=\frac{3-2f_n}{W}+\so{\frac{1}{\log(n)}},
    \end{split}
    \end{equation}
    where the second identity comes from \eqref{eq:Rn}.
    Similar calculations yield that
    \begin{equation}\label{eq:T2}
    \begin{split}
    &1-\frac{\lambda_{d_n-2}}{H_n}= \frac{3+2f_n}{\log(n)}+\so{\frac{1}{\log(n)}}.
    \end{split}
    \end{equation}
    Henceforth, from \eqref{eq:T1} and \eqref{eq:T2}, we conclude that for all $n$ big enough
    \begin{equation}\label{eq:u}
    \begin{split}
    &\min_{j\neq d_n-1,d_n,d_n+1}u_{j}=1-\max\curly{\frac{\lambda_{d_n-2}}{H_n};\frac{\lambda_{d_n+2}}{H_n}}\\
    &=\frac{3-2f_n}{\log(n)}+\so{\frac{1}{\log(n)}}\geq \frac{1}{2\log(n)}.
    \end{split}
    \end{equation}
    Then from the Markov inequality with $h=\frac{H_n}{\lambda_j}-1>0$, $a=\ln(1+h)=\ln(\frac{H_n}{\lambda_j})$ and from \eqref{eq:prop} we get that
    \begin{equation}\label{eq:Mar}
    \begin{split}
    &\Pbb{V_j>H_n}\leq \Ebb{e^{a V_j}}e^{-a H_n}=e^{\lambda_j\lbrb{e^{a}-1}-a H_n}=e^{\lambda_jh-H_n\ln(1+h)}\\
    &=e^{(H_n-\lambda_j)-H_n\ln\lbrb{\frac{H_n}{\lambda_j}}}=e^{H_n\lbrb{u_{j}+\ln(1-u_{j})}}\leq e^{-CH_nu^2_{j}},
    \end{split}
    \end{equation}
    where in the last inequality we have used \eqref{eq:u} and  the obvious fact that $\ln(1-x)+x\leq -Cx^2$, for some $C>0$ small enough. However, again from  \eqref{eq:u} and \eqref{eq:Rn},
    \[\min_{j\neq d_n-1,d_n,d_n+1}u_{j}\geq \frac{1}{2\log(n)}\text{ and } H_n=R_n-c\sqrt{R_n}\geq \frac{n}{\log^{2}(n)}\]
    for all $n\geq n^*\geq n(c)$. Therefore, for all $n\geq n^*\geq n(c)$ we have from \eqref{eq:Mar} that
    \begin{equation*}
    \begin{split}
    &\Pbb{V_j>H_n} \leq e^{-\frac{C}4\frac{H_n}{\log^2(n)}}\leq e^{-\frac{C}4\frac{n}{\log^4(n)}}.
    \end{split}
    \end{equation*}
    Let $U_n:=\lceil e^{\frac{C}4\frac{n}{\log^5(n)}} \rceil$. Then, clearly
    \begin{equation*}
    \begin{split}
    \limi{n}&\sum_{j=1;j\neq d_n-1,d_n,d_n+1}^{U_n}\Pbb{V_j>H_n}\leq \limi{n}U_ne^{-\frac{C}4\frac{n}{\log^4(n)}}=0;\\
    &\sum_{j=1;j\neq d_n-1,d_n,d_n+1}^{U_n}\Pbb{V_j>H_n}\leq\lceil e^{\frac{C}4\frac{n}{\log^5(n)}} \rceil\so{e^{-\frac{C}{4}\frac{n}{\log^4(n)}}} =\so{e^{-\frac{C}{8}\frac{n}{\log^4(n)}}}.
    \end{split}
    \end{equation*}
    Let us consider now $j>U_n$. Since $d_n\sim W\sim \log(n)$, see \eqref{d} and \eqref{w}, then $U_n-d_n>0$. We use the Chebyshev's inequality to get
    \begin{equation*}
    \begin{split}
    &\Pbb{V_j>H_n}\leq\frac{\lambda_j}{H_n}=    \frac{R_n\prod_{k=d_n+1}^{j}\frac{W}{k}}{H_n}=\frac{\prod_{k=d_n+1}^{j}\frac{W}{k}}{1-\frac{c}{R^{\frac12}_n}}\\
    &\leq   \frac{\prod_{k=U_n}^{j}\frac{W}{k}}{1-\frac{c}{R^{\frac12}_n}}\leq\frac{\frac{W^{j-U_n}}{U^{j-U_n}_n}}{1-\frac{c}{R^{\frac12}_n}}\sim  \frac{W^{j-U_n}}{U^{j-U_n}_n}.
    \end{split}
    \end{equation*}
    Summing over $j>U_n$, we get that
    \begin{equation*}
    \begin{split}
    &\limi{n}\sum_{j>U_n}\Pbb{V_j>H_n}\leq \limi{n}\sum_{j>U_n}\frac{W^{j-U_n}}{U^{j-U_n}_n}=\limi{n}\frac{W}{U_n}\frac{1}{1-\frac{W}{U_n}}=0;\\
    &\sum_{j>U_n}\Pbb{V_j>H_n}\lesssim \frac{W}{U_n}=\so{\log(n)e^{-\frac{C}{4}\frac{n}{\log^5(n)}}},
    \end{split}
    \end{equation*}
    since evidently $\frac{W}{U_n}=\so{\log(n)e^{-\frac{C}{4}\frac{n}{\log^5(n)}}}$. This concludes the proof of the proposition.
\end{proof}
As an important corollary we obtain that
\begin{corollary}\label{cor:prod}
    For any $c\in\Rb$, with $H_n:=R_n-c\sqrt{R_n}$
    \begin{equation*}
    \begin{split}
    \limi{n}    &   \prod_{j=1,j\neq d_n-1,d_n,d_n+1}^{\infty}\Pbb{V_j\leq R_n-c\sqrt{R_n}}=1
    \end{split}
    \end{equation*}
    and even more
     \begin{equation*}
    \begin{split}
    &\prod_{j=1,j\neq d_n-1,d_n,d_n+1}^{\infty}\Pbb{V_j\leq R_n-c\sqrt{R_n}} =1-\so{e^{-\frac{n}{\log^6(n)}}}.
    \end{split}
    \end{equation*}
\end{corollary}
\begin{proof}
    The proof is trivial using Proposition \ref{prop:est}. First note that it additionally implies that
    \begin{equation*}
    \begin{split}
    &\limi{n}\sup_{j\geq 1;j\neq d_n-1,d_n,d_n+1}\Pbb{V_j>H_n}=0
    \end{split}
    \end{equation*}
     and then note that the latter together with $\log(1-x)=-x+\so{x}$, as $x\to 0$, triggers
    \begin{equation*}
    \begin{split}
    &       \limi{n}\prod_{j=1,j\neq d_n-1,d_n,d_n+1}^{\infty}\Pbb{V_j\leq R_n-c\sqrt{R_n}}\\
    &=\limi{n} e^{\sum_{j=1,j\neq d_n-1,d_n,d_n+1}^{\infty}\log(1-\Pbb{V_j>H_n})}\\
    &=\limi{n}  e^{-\lbrb{1+\soo{1}}\sum_{j=1,j\neq d_n-1,d_n,d_n+1}^{\infty}\Pbb{V_j>H_n}} =1
    \end{split}
    \end{equation*}
    and
    \begin{equation*}
    \begin{split}
    &   \prod_{j=1,j\neq d_n-1,d_n,d_n+1}^{\infty}\Pbb{V_j\leq R_n-c\sqrt{R_n}}= e^{\sum_{j=1,j\neq d_n-1,d_n,d_n+1}^{\infty}\log(1-\Pbb{V_j>H_n})}\\
    &=  e^{-\lbrb{1+\soo{1}}\sum_{j=1,j\neq d_n-1,d_n,d_n+1}^{\infty}\Pbb{V_j>H_n}} =1-\so{e^{-\frac{n}{\log^6(n)}}}.
    \end{split}
    \end{equation*}
\end{proof}

We now have all the ingredients for the proof of Theorem \ref{cor:max}.
\begin{proof}[Proof of Theorem \ref{cor:max}]
     For any $c\in \Rb$, we have from \eqref{eq:maxIm}, Corollary \ref{cor:prod} and \eqref{eq:Norm} that
    \begin{equation*}
    \begin{split}
    &\Pbb{\bar{V}_n\leq R_n-c\sqrt{R_n}}=\lbrb{\prod_{j=d_n-1}^{d_n+1}\Pbb{V_j\leq R_n-c\sqrt{R_n}}}\lbrb{1-\so{e^{-\frac{n}{\log^6(n)}}}},
    \end{split}
    \end{equation*}
    which proves \eqref{eq:maxIm3}. However, since $\lbrb{V_{d_n}-R_n}/\sqrt{R_n}$ converges in distribution to the standard normal law we immediately get that
    \begin{equation*}
    \begin{split}
    &\prod_{j=d_n-1}^{d_n+1}\Pbb{V_j\leq R_n-c\sqrt{R_n}}\\
    &=\Phi(-c)\Pbb{V_{d_n-1}\leq R_n-c\sqrt{R_n}}\Pbb{V_{d_n+1}\leq R_n-c\sqrt{R_n}}+\so{1},
    \end{split}
    \end{equation*}
    which settles \eqref{eq:maxIm4} and concludes the proof.
\end{proof}

Now we are ready to commence the proof of Theorem \ref{thm:main} for which we need a helpful claim. Let
\begin{equation}\label{eq:B}
b(W):=b(W(n))=W\lbrb{1+W}e^W=n(1+W)\stackrel{\eqref{w}}{\sim} n\log(n).
\end{equation}
Then recalling the definition of $M_n$ under the probability
measure $P$, see \eqref{rv}, we
have the following statement.
\begin{prop}\label{thm:asymp}
    The following asymptotic relation holds for any
    $c\in\Rb\cup\curly{-\infty}$:
    \begin{equation}\label{eq:lim}
    \begin{split}
    &\frac{B_n}{n!}P\lbrb{M_n\leq R_n-c\sqrt{R_n}}=\frac{e^{e^W-1}}{\sqrt{2\pi}}\frac{\prod_{j=d_n-1}^{d_n+1}\Pbb{V_j\leq R_n-c\sqrt{R_n}}}{W^n\sqrt{b(W)}}\lbrb{1+\so{1}}.
    \end{split}
    \end{equation}
    As a consequence
    \begin{equation}\label{eq:lim11}
    \begin{split}
    &P\lbrb{M_n\leq R_n-c\sqrt{R_n}}=\prod_{j=d_n-1}^{d_n+1}\Pbb{V_j\leq R_n-c\sqrt{R_n}}+\so{1}.
    \end{split}
    \end{equation}
\end{prop}

    We set  $m=\lfloor H_n\rfloor=\lfloor R_n-c\sqrt{R_n}\rfloor$ and use the decomposition \eqref{eq:decompo} and \eqref{eq:jnk} with $k=1$ to recall that
\begin{equation}\label{eq:J_1}
\begin{split}
&J_{1,n}=\frac{W^{-n}}{2\pi}\int_{-\delta_n}^{\delta_n} e^{e^{We^{i\theta}}-1}F_m\lbrb{We^{i\theta}}e^{-i\theta n}d\theta.
\end{split}
\end{equation}
Then we have the following result.
\begin{lemma}\label{lem:J1}
    We have that
    \begin{equation}\label{eq:J_111}
    \begin{split}
    &J_{1,n}=\frac{W^{-n}e^{e^W-1}}{\sqrt{2\pi}\sqrt{b(W)}}\lbrb{\prod_{j=d_n-1}^{d_n+1}\Pbb{V_j\leq m}+\so{1}}.
    \end{split}
    \end{equation}
\end{lemma}
    The proof of the lemma requires several intermediate results. The first provides a simplification of the infinite product of \eqref{eq:repF}. Recall that $d_n=\lfloor  W\rfloor$.
    \begin{prop}\label{prop:repF}
        We have that
        \begin{equation}\label{eq:repProd}
        \begin{split}
        \prod_{j=1,j\neq d_n,d_n-1,d_n+1}^{\infty}\Ebb{e^{i\theta j (V_j-\lambda_j)} \ind{V_j\leq m}}&=e^{\phi_n(\theta)}+K_{n,m}(\theta),
        \end{split}
        \end{equation}
        where
        \begin{equation}\label{eq:phin}
        \begin{split}
        &\phi_n(\theta)=e^{We^{i\theta}}-e^{W}-\sum_{j=d_n-1}^{d_n+1}\frac{W^je^{i\theta j}}{j!}+\sum_{j=d_n-1}^{d_n+1}\frac{W^j}{j!}-i\theta n+i\theta\sum_{j=d_n-1}^{d_n+1}j\lambda_j
        \end{split}
        \end{equation}
        and
        \begin{equation}\label{eq:estimate}
        \begin{split}
        &\sup_{\theta\in\Rb}\abs{K_{n,m}(\theta)}=\so{e^{-\frac{n}{\log^6(n)}}}=\so{\frac{1}{\sqrt{b(W)}}}.
        \end{split}
        \end{equation}
    \end{prop}
\begin{proof}[Proof of Proposition \ref{prop:repF}]
    From the independence of $\curly{V_j}_{j\geq 1}$ we observe that
    \begin{equation*}
    \begin{split}
    &\prod_{j=1,j\neq d_n,d_n-1,d_n+1}^{\infty}\Ebb{e^{i\theta j (V_j-\lambda_j)} \ind{V_j\leq m}}\\
    &=\Ebb{e^{i\theta\sum_{j=1,j\neq d_n,d_n-1,d_n+1}^{\infty}j(V_j-\lambda_j)}\ind{\bigcap_{j=1,j\neq d_n,d_n-1,d_n+1}^{\infty}\curly{V_j\leq m}} }\\
    &=\Ebb{e^{i\theta\sum_{j=1,j\neq d_n,d_n-1,d_n+1}^{\infty}j(V_j-\lambda_j)}}-\\
    &-\Ebb{e^{i\theta\sum_{j=1,j\neq d_n,d_n-1,d_n+1}^{\infty}j(V_j-\lambda_j)}\ind{\bigcup_{j=1,j\neq d_n,d_n-1,d_n+1}^{\infty}\curly{V_j> m}}}\\
    &=\prod_{j=1,j\neq d_n,d_n-1,d_n+1}^{\infty}\Ebb{e^{i\theta j (V_j-\lambda_j)}}+K_{n,m}(\theta).
    \end{split}
    \end{equation*}
    Clearly, from \eqref{eq:sum3} and $m=\lfloor H_n\rfloor$
    \begin{equation*}
    \begin{split}
    &\sup_{\theta\in\Rb}\abs{K_{n,m}(\theta)}\leq \sum_{j=1;j\neq d_n-1,d_n,d_n+1}\Pbb{V_j>m}\\
    &=\sum_{j=1;j\neq d_n-1,d_n,d_n+1}\Pbb{V_j>H_n}
    =\so{e^{-\frac{n}{\log^6(n)}}},
    \end{split}
    \end{equation*}
    which  with the help of \eqref{eq:B} proves \eqref{eq:estimate}. Next, note that since
    \begin{equation*}
    \begin{split}
    &\Ebb{e^{i\theta\sum_{j=1}^{\infty}j(V_j-\lambda_j)}}=\prod_{j=1}^\infty e^{-\frac{W^j}{j!}\lbrb{1-e^{i\theta j}}}e^{-i\theta j\lambda_j}=e^{-i\theta n}e^{e^{We^{i\theta}}-e^{W}},
    \end{split}
    \end{equation*}
    where we have used \eqref{eq:a}, we conclude that
    \begin{equation*}
    \begin{split}
    &\prod_{j=1,j\neq d_n,d_n-1,d_n+1}^{\infty}\Ebb{e^{i\theta j (V_j-\lambda_j)}}\\
    &=e^{e^{We^{i\theta}}-e^{W}-\sum_{j=d_n-1}^{d_n+1}\frac{W^je^{i\theta j}}{j!}+\sum_{j=d_n-1}^{d_n+1}\frac{W^j}{j!}-i\theta n+i\theta\sum_{j=d_n-1}^{d_n+1}j\lambda_j}
    =e^{\phi_n(\theta)}.
    \end{split}
    \end{equation*}
    Thus we obtain \eqref{eq:repProd} and \eqref{eq:phin} and the proposition is  therefore proved.
\end{proof}
Next, Proposition \ref{prop:repF} allows with an application of \eqref{eq:repF} in \eqref{eq:J_1}  the following new representation of $J_{1,n}$:
\begin{equation}\label{eq:J_11}
\begin{split}
&J_{1,n}=\frac{W^{-n}e^{e^W-1}}{2\pi}\int_{-\delta_n}^{\delta_n}    \prod_{j=d_n-1}^{d_n+1}\Ebb{e^{i\theta j (V_j-\lambda_j)} \ind{V_j\leq m}}\lbrb{e^{\phi_n(\theta)}+\so{e^{-\frac{n}{\log^6(n)}}}}d\theta\\
&=\frac{W^{-n}e^{e^W-1}}{2\pi}\lbrb{\int_{-\delta_n}^{\delta_n}     \prod_{j=d_n-1}^{d_n+1}\Ebb{e^{i\theta j (V_j-\lambda_j)} \ind{V_j\leq m}}e^{\phi_n(\theta)}d\theta+\so{e^{-\frac{n}{\log^6(n)}}}}.
\end{split}
\end{equation}
We proceed to give some elementary properties for $\phi_n$.
\begin{lemma}\label{lem:phin}
   We have that
    \begin{equation}\label{eq:phinder}
    \begin{split}
    \phi_n(0)&=\phi'_n(0)=0,\\ \phi_n''(0)&=-b(W)+\sum_{j=d_n-1}^{d_n+1}j^2\frac{W^j}{j!}=-b(W)\lbrb{1+\so{1}},\\
    \abs{\phi'''_n(\theta)}&\leq  \lbrb{W^3+3W^2+W}e^{W}+\sum_{j=d_n-1}^{d_n+1} j^2\frac{W^j}{j!},\\
    &= W^3e^W\lbrb{1+\so{1}}= n\log^2(n)\lbrb{1+\so{1}}.
    \end{split}
    \end{equation}
    Therefore, for $\abs{\theta}\leq \delta_n$ we have that
    \begin{equation}\label{eq:Taylor}
    \begin{split}
    &\phi_n(\theta)=-b(W)\frac{\theta^2}2(1+\so{1})+\so{1}.
    \end{split}
    \end{equation}
\end{lemma}
\begin{proof}
    For the first derivative of $\phi_n$ we get
    \begin{equation*}
    \begin{split}
    &\phi'_n(\theta)=iWe^{i\theta}e^{We^{i\theta}}-i\sum_{j=d_n-1}^{d_n+1} j\frac{W^j}{j!}e^{i\theta j}-in+i\sum_{j=d_n-1}^{d_n+1}j\lambda_j;\\
    &\phi'_n(0)=iWe^W-i\sum_{j=d_n-1}^{d_n+1} j\frac{W^j}{j!}-in+i\sum_{j=d_n-1}^{d_n+1}j\frac{W^j}{j!}=0,
    \end{split}
    \end{equation*}
    where the very last relation follows from \eqref{eq}.
    Also
    \begin{equation*}
    \begin{split}
    \phi''_n(\theta)&=-W^2e^{2i\theta}e^{We^{i\theta}}-We^{i\theta}e^{We^{i\theta}}+\sum_{j=d_n-1}^{d_n+1} j^2\frac{W^j}{j!}e^{i\theta j};\\
    \phi''_n(0)&=-W^2e^W-We^W+\sum_{j=d_n-1}^{d_n+1} j^2\frac{W^j}{j!}\\
    &=-b(W)+\sum_{j=d_n-1}^{d_n+1} j^2\frac{W^j}{j!},
    \end{split}
    \end{equation*}
    where the last identity holds since \eqref{eq:B} is valid.
    In addition, note that since $\max_{j\geq 1} \lambda_j=\max_{j\geq 1} \frac{W^j}{j!}=R_n$, see \eqref{eq:maxPar}, we have that
    \begin{equation*}
    \begin{split}
    &\sum_{j=d_n-1}^{d_n+1}j^2\frac{W^j}{j!}\leq 3(d_n+1)^2R_n= 3\log^2(n)\frac{n}{\sqrt{2\pi}\log^{3/2}(n)}\lbrb{1+\so{1}}\\
    &=\frac{3}{\sqrt{2\pi}}n\log^{\frac12}(n)\lbrb{1+\so{1}}=\so{b(W)},
    \end{split}
    \end{equation*}
    where the first equality follows from \eqref{eq:Rn} and the very last relation is due to \eqref{eq:B}.
    The first two parts of \eqref{eq:phinder} are therefore established.
    Taking third derivative we get that
    \begin{equation*}
    \begin{split}
    &\abs{\phi'''_n(\theta)}\leq \lbrb{W^3+3W^2+W}e^{W}+\sum_{j=d_n-1}^{d_n+1} j^3\frac{W^j}{j!}\\
    &= W^3e^W\lbrb{1+\so{1}}+\sum_{j=d_n-1}^{d_n+1} j^3\frac{W^j}{j!}.
    \end{split}
    \end{equation*}
    However, using \eqref{eq:Rn}, \eqref{eq:maxPar} and \eqref{w} we deduce that
    \begin{equation*}
    \begin{split}
    &W^3e^W=nW^2= n\log^2(n)\lbrb{1+\so{1}}\\
    &\sum_{j=d_n-1}^{d_n+1}j^3\frac{W^j}{j!}\leq 3(d_n+1)^3R_n\\
    &= 3\log^3(n)\frac{n}{\log^{3/2}(n)}\lbrb{1+\so{1}}\\
    &= 3n\log^{3/2}(n)\lbrb{1+\so{1}}.
    \end{split}
    \end{equation*}
    This confirms the last part of \eqref{eq:phinder}. Relation \eqref{eq:Taylor} comes with the help of  \eqref{eq:phinder}, Taylor expansion of third order, $\abs{\theta}\leq \delta_n$, see \eqref{eq:delta_n}, and
    \begin{equation*}
    \begin{split}
    \phi_n(\theta)&=-b(W)\frac{\theta^2}2(1+\so{1})+\bo{\delta^3_n n\log^2(n)}\\
    &=-b(W)\frac{\theta^2}2(1+\so{1})+\bo{ \frac{n^{\frac{3}{7}}}{n^{\frac{3}{2}}\log^{\frac{3}{2}}(n)}n\log^2(n)}\\
    &=-b(W)\frac{\theta^2}2(1+\so{1})+\so{1}.
    \end{split}
    \end{equation*}
    This concludes the overall proof of Lemma \ref{lem:phin}.
\end{proof}
We are now ready to tackle Lemma \ref{lem:J1}.
\begin{proof}[Proof of Lemma \ref{lem:J1}]

We write \[\delta'_n=\frac{\log\log(n)}{\sqrt{n\log(n)}}.\]  Next for $\lbbrbb{\delta'_n,\delta_n}$ we get from the representation \eqref{eq:J_11} and the asymptotic relation \eqref{eq:Taylor} that
\begin{equation*}
\begin{split}
&\abs{\int_{|\theta|\in\lbbrbb{\delta'_n,\delta_n}}\prod_{j=d_n,d_n-1,d_n+1}\Ebb{e^{i\theta j (V_j-\lambda_j)}\ind{V_j\leq m}}e^{\phi_n(\theta)}d\theta}\\
&=\abs{\int_{|\theta|\in\lbbrbb{\delta'_n,\delta_n}}\prod_{j=d_n,d_n-1,d_n+1}\Ebb{e^{i\theta j (V_j-\lambda_j)} \ind{V_j\leq m}}e^{-b(W)\frac{\theta^2}2(1+\soo{1})+\soo{1}}d\theta}\\
&\leq\int_{|\theta|\in\lbbrbb{\delta'_n,\delta_n}}e^{-b(W)\frac{\theta^2}2(1+\soo{1})+\soo{1}}d\theta \\ &=\frac{1}{\sqrt{b(W)}}\int_{|y|\in\lbbrbb{\delta'_n\sqrt{b(W)},\delta_n\sqrt{b(W)}}}e^{-\frac{y^2}{2}\lbrb{1+\soo{1}}+\soo{1}}dy\\
&\leq \frac{1}{\sqrt{b(W)}}\int_{|y|\geq \delta'_n\sqrt{b(W)}}e^{-\frac{y^2}{2}\lbrb{1+\soo{1}+\soo{1}}}dy\\
&=\frac{\so{1}}{\sqrt{b(W)}},
\end{split}
\end{equation*}
because
\[\limi{n}\delta'_n\sqrt{b(W)}=\limi{n}\frac{\log\log(n)}{\sqrt{n\log(n)}}\sqrt{n(1+W)}=\infty.\]
Therefore, we conclude from \eqref{eq:J_11} combined with
\eqref{eq:B} and the fact that
$e^{-\frac{n}{\log^6(n)}}$ is of lower order than
$\frac{1}{\sqrt{b(W)}}$ that
\begin{equation}\label{eq:J_12}
\begin{split}
&J_{1,n}=\frac{W^{-n}e^{e^W-1}}{2\pi}\!\!\lbrb{\int_{-\delta'_n}^{\delta'_n}    \prod_{j=d_n-1}^{d_n+1}\!\!\!\!\!\Ebb{e^{i\theta j (V_j-\lambda_j)}\ind{V_j\leq m}}e^{\phi_n(\theta)}d\theta+\so{\frac{1}{\sqrt{b(W)}}}}\!\!.
\end{split}
\end{equation}
It remains to examine the integral term. Using \eqref{eq:Taylor} we arrive as above at the equivalent term
\begin{equation*}
\begin{split}
&\int_{-\delta'_n}^{\delta'_n}  \prod_{j=d_n-1}^{d_n+1}\Ebb{e^{i\theta j (V_j-\lambda_j)} \ind{V_j\leq m}}e^{-b(W)\frac{\theta^2}2(1+\soo{1})+\soo{1}}d\theta\\
&=\frac{1}{\sqrt{b(W)}}\int_{-\delta'_n\sqrt{b(W)}}^{\delta'_n\sqrt{b(W)}}\prod_{j=d_n-1}^{d_n+1}\!\!\!\!\Ebb{e^{i\frac{y}{\sqrt{b(W)}}j\sqrt{\lambda_{j}} \frac{(V_j-\lambda_{j})}{\sqrt{\lambda_{j}}}}\ind{V_j\leq m}}e^{-\frac{y^2}2(1+\soo{1})+\soo{1}}dy.
\end{split}
\end{equation*}
Now, from \eqref{w}, \eqref{eq:maxPar} and \eqref{eq:Rn}, we have that with $z_j(y):=\frac{y}{\sqrt{b(W)}}j\sqrt{\lambda_{j}}, y\in\lbbrbb{-\delta'_n\sqrt{b(W)},\delta'_n\sqrt{b(W)}}$ the following bound is valid
\begin{equation*}
\begin{split}
&\abs{z_j(y)}\leq \delta'_n (d_n+1)\sqrt{R_n}\sim  \frac{\log\log(n)}{\sqrt{n\log(n)}}\log^{1/4}(n)\sqrt{n}=\so{1}.
\end{split}
\end{equation*}
Henceforth,
\begin{equation*}
\begin{split}
&\prod_{j=d_n-1}^{d_n+1}\Ebb{e^{iz_j(y)\frac{(V_j-\lambda_{j})}{\sqrt{\lambda_{j}}}} \ind{V_j\leq m}}\\
&=\prod_{j=d_n-1}^{d_n+1}\lbrb{\Pbb{V_j\leq m}+\Ebb{\lbrb{e^{iz_j(y) \frac{(V_j-\lambda_{j})}{\sqrt{\lambda_j}}}-1}\ind{V_j\leq m}}}\\
&=\prod_{j=d_n-1}^{d_n+1}\Pbb{V_j\leq m}+\so{1},
\end{split}
\end{equation*}
 where the very last relation follows from
\begin{equation*}
\begin{split}
&\abs{\Ebb{\lbrb{e^{iz_j(y) \frac{(V_j-\lambda_{j})}{\sqrt{\lambda_{j}}}}-1}\ind{V_j\leq m}}}\leq \Ebb{\min\curly{1,|z_j(y)| \abs{\frac{V_j-\lambda_{j}}{\sqrt{\lambda_{j}}}}}}=\so{1},
\end{split}
\end{equation*}
because $|z_j(y)|=\so{1}$ uniformly for the specified range of
$y$, \[\limi{n}\frac{(V_j-\lambda_{j})}{\sqrt{\lambda_{j}}}=Z_1\]
 and the sequence
$\curly{\frac{V_j-\lambda_{j}}{\sqrt{\lambda_{j}}}}_{j\geq 1}$ is
as a consequence tight. Therefore, we further conclude using
$\limi{n}\delta'_n\sqrt{b(W)}=\infty$ that
\begin{equation*}
\begin{split}
&\int_{-\delta'_n}^{\delta'_n}  \prod_{j=d_n-1}^{d_n+1}\Ebb{e^{i\theta j (V_j-\lambda_j)}\ind{V_j\leq m}}e^{-b(W)\frac{\theta^2}2(1+\soo{1})+\soo{1}}d\theta\\
&=\frac{\prod_{j=d_n-1}^{d_n+1}\Pbb{V_j\leq m}}{\sqrt{b(W)}}\int_{-\delta'_n\sqrt{b(W)}}^{\delta'_n\sqrt{b(W)}}e^{-\frac{y^2}2(1+\soo{1})+\soo{1}}dy+\so{\frac{1}{\sqrt{b(W)}}}\\
&=\sqrt{2\pi}\frac{\prod_{j=d_n-1}^{d_n+1}\Pbb{V_j\leq m}}{\sqrt{b(W)}}+\so{\frac{1}{\sqrt{b(W)}}}.
\end{split}
\end{equation*}
This together with \eqref{eq:J_12}  proves \eqref{eq:J_111} and Lemma \ref{lem:J1} is completed.
\end{proof}

 We recall that $\gamma_n:=\frac{1}{\log^{\frac{1}{5}}(n)}$, see \eqref{eq:gamma n}, and
consider
\begin{equation}\label{eq:J_2}
\begin{split}
&J_{2,n}=\frac{W^{-n}}{2\pi}\int_{\pi\geq |\theta|\geq \gamma_n} e^{e^{We^{i\theta}}-1}F_m\lbrb{We^{i\theta}}e^{-i\theta n}d\theta\\
&=\frac{W^{-n}e^{e^W-1}}{2\pi}\int_{\pi\geq|\theta|\geq \gamma_n}    \prod_{j=d_n-1}^{d_n+1}\Ebb{e^{i\theta j (V_j-\lambda_j)}\ind{V_j\leq m}}e^{\phi_n(\theta)}d\theta\\
& \,+\frac{W^{-n}e^{e^W-1}}{2\pi}\so{\frac{1}{\sqrt{b(W)}}},
\end{split}
\end{equation}
see \eqref{eq:jnk} and recall the claims of Proposition \ref{prop:repF}  for the last identity. To investigate the integral term we discuss further $\phi_n$.
\begin{lemma}\label{lem:phin1}
   We have that
    \begin{equation}\label{eq:Rephi}
    \begin{split}
    &\Re\,\phi_n(\theta)=e^{W\cos(\theta)}\cos(W\sin(\theta))-e^W+\sum_{j=d_n-1}^{d_n+1}\frac{W^j}{j!}\lbrb{1-\cos(\theta j)}\\
    &=\lbrb{e^{W\cos(\theta)}-e^W}-e^{W\cos(\theta)}\lbrb{1-\cos\lbrb{W\sin(\theta)}}+\sum_{j=d_n-1}^{d_n+1}\frac{W^j}{j!}\lbrb{1-\cos(\theta j)}.
    \end{split}
    \end{equation}
    As a consequence the following estimate holds
    \begin{equation}\label{eq:Reph}
    \begin{split}
    & \limi{n}\sup_{\pi\geq |\theta|\geq \gamma_n}\log(n)\frac{\Re\,\phi_n(\theta)}{n}\leq -C<0.
    \end{split}
    \end{equation}
\end{lemma}
\begin{proof}[Proof of Lemma \ref{lem:phin1}]
    Relation \eqref{eq:Rephi} follows easily from \eqref{eq:phin}.  Since the second term in the last relation of \eqref{eq:Rephi} is non-positive and $\cos(\gamma_n)\geq \cos(\theta), \gamma_n\leq |\theta|\leq \pi,$ we get that for all large $n$
   \begin{equation*}
   \begin{split}
    &\Re\,\phi_n(\theta)\leq \lbrb{e^{W\cos(\gamma_n)}-e^W}+ 6R_n= -e^{W}\lbrb{1-e^{-W(1-\cos(\gamma_n))}}+ 6R_n\\
    &\leq -e^W\lbrb{1-e^{-W\frac{\gamma^2_n}{2}}}+ 6R_n\\
    &= -e^W\lbrb{1-e^{-\frac{W}{2\log^{\frac25}(n)}}}+6R_n=-\frac{n}{W}\lbrb{1-e^{-\frac{W}{2\log^{\frac25}(n)}}}+6R_n\\
    &=-\frac{n}{W}\lbrb{1+\so{1}}+6R_n= \lbrb{-\frac{n}{ \log(n)}+\frac{n}{\sqrt{2\pi}\log^{\frac{3}{2}}(n)}}\lbrb{1+\so{1}}\\
    &=-\frac{n}{\log(n)}\lbrb{1-\frac{4}{\log^{\frac{1}{2}}(n)}}\lbrb{1+\so{1}}\\
    &=-\frac{n}{\log(n)}\lbrb{1+\so{1}},
    \end{split}
    \end{equation*}
    where in the first inequality we have used from \eqref{eq:maxPar} that
    \[\sum_{j=d_n-1}^{d_n+1}\frac{W^j}{j!}\lbrb{1-\cos(\theta j)}\leq 6R_n,\]
    in the second one that $1-\cos(x)\leq \frac{x^2}{2}$ and  we have also invoked \eqref{eq} and \eqref{eq:Rn}. This shows \eqref{eq:Reph} and concludes the proof of the statement.
\end{proof}
Therefore with Lemma \ref{lem:phin1} we conclude that for all large $n$
\begin{equation*}
\begin{split}
&\abs{\int_{\pi\geq |\theta|\geq \gamma_n}\prod_{j=d_n,d_n-1,d_n+1}\Ebb{e^{i\theta j (V_j-\lambda_j)} \ind{V_j\leq m}}e^{\phi_n(\theta)}d\theta}\leq \int_{\pi\geq |\theta|\geq \gamma_n}e^{\Re\,\phi_n(\theta)}d\theta\\
&\leq 2\pi  exp\curly{ -C\frac{n}{
\log(n)}}=\so{\frac{1}{\sqrt{b(W)}}},
\end{split}
\end{equation*}
 where the very last relation follows from \eqref{eq:B}. Henceforth,
 applying
this in \eqref{eq:J_2} we arrive at
the following
\begin{lemma}\label{lem:J_2}
    We have that
    \begin{equation}\label{eq:J_21}
    \begin{split}
    J_{2,n}&=\frac{W^{-n}}{2\pi}\int_{\pi\geq |\theta|\geq \gamma_n} e^{e^{We^{i\theta}}-1}F_m\lbrb{We^{i\theta}}e^{-i\theta n}d\theta\\
    &=\frac{W^{-n}e^{e^W-1}}{2\pi}\so{\frac{1}{\sqrt{b(W)}}}.
    \end{split}
    \end{equation}
\end{lemma}
It remains to consider the region
$|\theta|\in\lbbrbb{\delta_n,\gamma_n}$, see \eqref{eq:delta_n}
for $\delta_n$ and \eqref{eq:gamma n} for $\gamma_n$. Recall
from \eqref{eq:jnk} that
\begin{equation}\label{eq:J_3}
\begin{split}
J_{3,n}&=\frac{W^{-n}}{2\pi}\int_{|\theta|\in\lbbrbb{\delta_n,\gamma_n}} e^{e^{We^{i\theta}}-1}F_m\lbrb{We^{i\theta}}e^{-i\theta n}d\theta\\
&=\frac{W^{-n}e^{e^W-1}}{2\pi}\int_{|\theta|\in\lbbrbb{\delta_n,\gamma_n}}\prod_{j=1}^{\infty}\Ebb{e^{i\theta j (V_j-\lambda_j)}\ind{V_j\leq m}}d\theta,
\end{split}
\end{equation}
where the latter follows from \eqref{eq:repF}. This estimate hinges upon the following elementary bounds.
\begin{lemma}\label{lem:prod3}
    For any positive integer $l<d_n-1$,
    \begin{equation}\label{eq:prod3}
    \begin{split}
    &\abs{\prod_{j=1}^{\infty}\Ebb{e^{i\theta j (V_j-\lambda_j)} \ind{V_j\leq m}}}\leq  e^{-\frac{W^{l}}{l!}\lbrb{1-\cos(\theta l)}}+\so{e^{-\frac{n}{\log^6(n)}}}.
    \end{split}
    \end{equation}
   Moreover, for $|\theta|\in\lbbrbb{\delta_n,\gamma_n}$ we have that
    \begin{equation}\label{eq:prod31}
    \begin{split}
    &\abs{\prod_{j=1}^{\infty}\Ebb{e^{i\theta j (V_j-\lambda_j)}\ind{V_j\leq m}}}=\so{\frac{1}{\sqrt{b(W)}}}.
    \end{split}
    \end{equation}
\end{lemma}
\begin{proof}
    From \eqref{eq:sum3} it follows that, for any $l<d_n-1$,
    \begin{equation*}
    \begin{split}
    &\abs{\prod_{j=1}^{\infty}\Ebb{e^{i\theta j (V_j-\lambda_j)}\ind{V_j\leq m}}}\leq \abs{\Ebb{e^{i\theta l (V_l-\lambda_l)}\ind{V_l\leq m}}}\\
    &=   \abs{\Ebb{e^{i\theta l (V_l-\lambda_l)}}}+  \so{e^{-\frac{n}{\log^6(n)}}}\\
    &=  e^{-\frac{W^{l}}{l!}\lbrb{1-\cos(\theta l)}}+\so{e^{-\frac{n}{\log^6(n)}}}
    \end{split}
    \end{equation*}
    or the claim \eqref{eq:prod3} follows. To derive \eqref{eq:prod31} we set \[\delta_n<\varphi_n:=\frac{1}{\log^2(n)}<\gamma_n.\]
     For $|\theta|\in \lbbrbb{\varphi_n,\gamma_n}$, we apply \eqref{eq:prod3} with $l=10$ to get
    \begin{equation*}
    \begin{split}
    &\abs{\prod_{j=1}^{\infty}\Ebb{e^{i\theta j (V_j-\lambda_j)} \ind{V_j\leq m}}}\leq  e^{-\frac{W^{10}}{10!}\lbrb{1-\cos(10\theta )}}+\so{e^{-\frac{n}{\log^6(n)}}}\\
    &\leq e^{-C100\frac{W^{10}}{10!}\varphi^2_n}+\so{e^{-\frac{n}{\log^6(n)}}}=e^{-C100\frac{W^{10}}{10!}\frac{1}{\log^{4}(n)}}+\so{e^{-\frac{n}{\log^6(n)}}}\\
    &=\so{\frac{1}{\sqrt{b(W)}}},
    \end{split}
    \end{equation*}
    where we have used that $1-\cos(x)\geq Cx^2$ for all $x$ small enough, the obvious convergence $10\varphi_n\to 0$ and the fact that
    from \eqref{eq:B} it follows that
    \[e^{-C100\frac{W^{10}}{10!}\frac{1}{\log^{4}(n)}}=\so{e^{-\log^5(n)}}=\so{\frac{1}{\sqrt{b(W)}}},\]
    where $b(W)\sim n\log(n)$, see \eqref{eq:B}. Finally, it remains to consider the case $|\theta|\in \lbbrbb{\delta_n,\varphi_n}$. In this case we set $l_n:=\lfloor a W\rfloor<d_n-1$ for some $a\in\lbrb{0,1}$. Then
    \[\limi{n}\abs{\theta l_n}\leq \limi{n}\varphi_n aW=a\limi{n}\frac{\log(n)}{\log^2(n)}=0,\]
    where in the first identity we have employed \eqref{w}.
    Therefore, the bound above can be replicated  with $\delta_n$ for $\varphi_n$ to derive
    \begin{equation*}
    \begin{split}
    &\abs{\prod_{j=1}^{\infty}\Ebb{e^{i\theta j (V_j-\lambda_j)}\ind{V_j\leq m}}}\leq   e^{-Cl^2_n\frac{W^{l_n}}{l_n!}\delta^2_n}+\so{e^{-\frac{n}{\log^6(n)}}}.
    \end{split}
    \end{equation*}
    Now, as $l_n\to\infty$, from Stirling approximation, see \eqref{eq:Stir}, the asymptotic \eqref{w},  and the expressions for $\delta_n$ and $l_n$ we get that
    \begin{equation*}
    \begin{split}
    &l^2_n\frac{W^{l_n}}{l_n!}\delta^2_n   =\frac{a^{\frac32}}{\sqrt{2\pi}} \lbrb{\frac{W}{l_n}}^{l_n}W^{\frac{3}{2}}e^{l_n}\frac{n^{\frac27}}{n\log(n)}\lbrb{1+\so{1}}\\
    &= \frac{a^{\frac32}}{\sqrt{2\pi}}\lbrb{\frac{W}{l_n}}^{l_n}\frac{e^{l_n}\sqrt{\log(n)}}{n^{\frac57}}\lbrb{1+\so{1}}\\
    &\geq \frac{e^{-1}a^{\frac32}}{\sqrt{2\pi}}\lbrb{\frac{W}{l_n}}^{l_n}\frac{e^{a\log(n)}\sqrt{\log(n)}}{n^{\frac57}}\lbrb{1+\so{1}}\\
    &\geq  \frac{e^{-1}a^{\frac32}}{\sqrt{2\pi}}\lbrb{\frac{1}{a}}^{l_n}\frac{\sqrt{\log(n)}}{n^{\frac57-a}}\lbrb{1+\so{1}}\\
    &\geq \frac{e^{-1}a^{\frac32}}{\sqrt{2\pi}}\frac{\sqrt{\log(n)}}{n^{\frac57-a}}\lbrb{1+\so{1}}.
    \end{split}
    \end{equation*}
    Choosing $a=\frac67$, we get the following estimate with some
    $C'>0$:
    \begin{equation*}
    \begin{split}
    &\abs{\prod_{j=1}^{\infty}\Ebb{e^{i\theta j (V_j-\lambda_j)}\ind{V_j\leq m}}}\leq   e^{-C'n^{\frac17}\sqrt{\log(n)}}+\so{e^{-\frac{n}{\log^6(n)}}}=\so{\frac{1}{\sqrt{b(W)}}}.
    \end{split}
    \end{equation*}
    This proves \eqref{eq:prod31} and concludes the claim.
\end{proof}
From Lemma \ref{lem:prod3} and \eqref{eq:prod31} we easily arrive at
\begin{equation}\label{eq:J_31}
\begin{split}
&J_{3,n}=\frac{W^{-n}e^{e^W-1}}{2\pi}\so{\frac{1}{\sqrt{b(W)}}}.
\end{split}
\end{equation}
Now, we are ready to complete the proof of Theorem \ref{thm:asymp}.
\begin{proof}[Proof of Theorem \ref{thm:asymp}]
     Clearly, $J_n=J_{1,n}+J_{2,n}+J_{3,n}$, see \eqref{eq:integral}, \eqref{eq:J_1}, \eqref{eq:J_2} and \eqref{eq:J_3}. However, the asymptotic relations \eqref{eq:J_111}, \eqref{eq:J_21} and \eqref{eq:J_31} yield immediately \eqref{eq:lim} for any $c\in\Rb\cup\curly{-\infty}$. Applying \eqref{eq:lim} with $c=-\infty$ we get that other then the product term the rest in \eqref{eq:lim} is the asymptotic behavior of $B_n/n!$. This deduces \eqref{eq:lim11}.
\end{proof}

Now we are in a position to prove Theorem \ref{thm:main}.
\begin{proof}[Proof of Theorem \ref{thm:main}]
    From \eqref{eq:lim11} of Proposition \ref{thm:asymp} and \eqref{eq:maxIm4} of Theorem \ref{cor:max} we deduce that for any $c\in\Rb$
    \begin{equation}\label{eq:lim12}
    \begin{split}
    &P\lbrb{M_n\leq R_n-c\sqrt{R_n}}=\Phi(-c)\prod_{j=d_n-1;d_n+1}\Pbb{V_j\leq R_n-c\sqrt{R_n}}+\so{1}.
    \end{split}
    \end{equation}
    Let us investigate the asymptotic of the other two terms in \eqref{eq:lim12}.
    Consider $j=d_n+1$ and note that \[T_n:=\frac{\lbrb{V_{d_n+1}-\lambda_{d_n+1}}}{\sqrt{\lambda_{d_n+1}}}\to Z_1.\]
    Therefore, we have that
    \begin{equation}\label{eq:afterMax}
    \begin{split}
    &\Pbb{V_{d_n+1}\leq R_n-c\sqrt{R_n}}=\Pbb{\frac{V_{d_n+1}-\lambda_{d_n+1}}{\sqrt{\lambda_{d_n+1}}}\leq \frac{R_n-\lambda_{d_n+1}}{\sqrt{\lambda_{d_n+1}}}-c\frac{\sqrt{R_n}}{\sqrt{\lambda_{d_n+1}}}}\\
    &=\Pbb{T_n\leq \frac{R_n-\lambda_{d_n+1}}{\sqrt{\lambda_{d_n+1}}}-c\frac{\sqrt{R_n}}{\sqrt{\lambda_{d_n+1}}}}.
    \end{split}
    \end{equation}
    Since  $T_n\to Z_1$ it suffices to understand the behavior of
    \begin{equation*}
    \begin{split}
    &\frac{R_n-\lambda_{d_n+1}}{\sqrt{\lambda_{d_n+1}}}-c\frac{\sqrt{R_n}}{\sqrt{\lambda_{d_n+1}}}.
    \end{split}
    \end{equation*}
    Feeding \eqref{eq:Rn} into \eqref{eq:see1} and applying \eqref{eq:ration} we arrive at
    \begin{equation}\label{eq:see2}
    \begin{split}
    \lbrb{\frac{R_n-\lambda_{d_n+1}}{\sqrt{\lambda_{d_n+1}}}-c\frac{\sqrt{R_n}}{\sqrt{\lambda_{d_n+1}}}}&=\sqrt{R_n}\frac{1-f_n}{\log(n)}\lbrb{1+\so{1}}-c+\so{1}\\
    &=\lbrb{\frac{1}{2\pi}}^\frac14\frac{\sqrt{n}}{\log^{\frac{7}{4}}(n)}\lbrb{1-f_n}\lbrb{1+\so{1}}-c+\so{1}.
    \end{split}
    \end{equation}
    Until the end we will work over subsequences but for clarity we will preserve the notation, e.g. $d_n+1$ instead of $d_{n_k}+1$ etc. Considering \eqref{eq:see2} we see from \eqref{eq:afterMax} that for any $u\in\lbbrbb{0,\infty}$ over a subsequence
    \begin{equation}\label{eq:afterMax1}
    \begin{split}
    &\limi{n}\Pbb{V_{d_n+1}\leq R_n-c\sqrt{R_n}}=\limi{n}\Pbb{T_n\leq \frac{R_n-\lambda_{d_n+1}}{\sqrt{\lambda_{d_n+1}}}-c\frac{\sqrt{R_n}}{\sqrt{\lambda_{d_n+1}}}}\\
    &=\Phi(u-c)\\
    &\iff \limi{n}\lbrb{\frac{R_n-\lambda_{d_n+1}}{\sqrt{\lambda_{d_n+1}}}-c\frac{\sqrt{R_n}}{\sqrt{\lambda_{d_n+1}}}}=u-c\\
    &\iff \limi{n}\lbrb{\frac{1}{2\pi}}^\frac14\frac{\sqrt{n}}{\log^{\frac{7}{4}}(n)}\lbrb{1-f_n}=u.
    \end{split}
    \end{equation}
    In the same vein we derive for any $u\in\lbbrbb{0,\infty}$ over a subsequence that
    \begin{equation}\label{eq:see5}
    \begin{split}
    &\lbrb{\frac{R_n-\lambda_{d_n-1}}{\sqrt{\lambda_{d_n-1}}}-c\frac{\sqrt{R_n}}{\sqrt{\lambda_{d_n-1}}}}=\lbrb{\frac{1}{2\pi}}^\frac14\frac{\sqrt{n}}{\log^{\frac{7}{4}}(n)}f_n\lbrb{1+\so{1}}-c+\so{1}\\
    &\limi{n}\Pbb{V_{d_n-1}\leq R_n-c\sqrt{R_n}}=\Phi(u-c)\\
    &\iff \limi{n}\lbrb{\frac{1}{2\pi}}^\frac14\frac{\sqrt{n}}{\log^{\frac{7}{4}}(n)}f_n=u.
    \end{split}
    \end{equation}
    From \eqref{eq:lim12},\eqref{eq:afterMax1} and \eqref{eq:see5} we then conclude the second item (\textit{ii}) of Theorem \ref{thm:main} over subsequences for which  \[\limi{n}\min\curly{f_n,1-f_n}\frac{\sqrt{n}}{\log^{\frac{7}{4}}(n)}=\limi{n}\vartheta_n\frac{\sqrt{n}}{\log^{\frac{7}{4}}(n)}=\infty.\]
    However, since \eqref{eq:afterMax1} and \eqref{eq:see5} cannot  hold simultaneously over a subsequence for finite values of $u$ we deduce the other two items of Theorem \ref{thm:main}.
    It remains to confirm that all items of Theorem \ref{thm:main} are possible. Clearly, the second one is attainable since $W$ grows logarithmically
    and fills the intervals between integers in a denser and denser way. Therefore, for a subsequence $\curly{n_k}_{k\geq 1}$ such that $\limi{k}\abs{\vartheta_{n_k} -1/2}=0$ we have that item (\textit{ii}) of Theorem \ref{thm:main} can materialize.
    Next,  one ought to understand how well $W$ is approximated by integers. This may be a difficult task in general. We will see that over a subsequence the last limit in \eqref{eq:afterMax1}  may fail. For this purpose we consider $W$. Clearly,
    \[x_m=me^{m},m\geq 1, \]
    are such that
    \[W(x_m)=m.\]
    Then for any $n\in\lbrb{x_m,x_{m+1}}$ we have that
    \[f_n=W(n)-\lfloor W(n)\rfloor=W(n)-W(x_m)=1-W(x_{m+1})+W(n)\]
    or
    \begin{equation*}
    \begin{split}
    &1-f_n=W(x_{m+1})-W(n)=\int_{n}^{x_{m+1}}W'(y)dy.
    \end{split}
    \end{equation*}
    From \eqref{eq} valid for any $x>0$ it follows that $W'(x)\sim x^{-1}$, as $x\to\infty$, and therefore setting $n_m=\lfloor x_{m+1}\rfloor$ and $v_m=x_{m+1}-n_m$ we have as $m\to\infty$, that
    \begin{equation*}
    \begin{split}
    &1-f_{n_m}=W(x_{m+1})-W(n_m)=\int_{n_m}^{x_{m+1}}W'(y)dy\sim \log\lbrb{\frac{x_{m+1}}{n_m}}\sim \frac{v_m}{n_m}.
    \end{split}
    \end{equation*}
    This fed back into the final term in \eqref{eq:afterMax1} over the subsequence $n_m$ yields that
    \begin{equation*}
    \begin{split}
    \limi{m}\frac{\sqrt{n_m}}{\log^{\frac{7}{4}}(n_m)}\lbrb{1-f_{n_m}}= \limi{m}\frac{v_m}{\sqrt{n_m}\log^{\frac{7}{4}}(n_m)}=0.
    \end{split}
    \end{equation*}
    This concludes the proof of the theorem.
\end{proof}


\begin{thebibliography}{15}

\bibitem{ABT03}
R. Arratia, A. D. Barbour and S. Tavar\'{e}. {\it Logarithmic
Combinatorial Structures: a Probabilistic Approach}, Europ. Math.
Soc., Z\"{u}rich, 2003.

\bibitem{AT94}
R. Arratia and S. Tavar\'{e}. Independent process approximations
for random combinatorial structures. Adv. Math. 104(1994), 90-154.

\bibitem{D58}
N. G. De Bruijn. {\it Asymptotic Methods in Analysis},
North-Holland, Amsterdam, 1958.

\bibitem{DP83}
J. M. De Laurentis and B. Pittel. Counting subsets of the random
partitions and the "Brownian Bridge" process. Stochastic Processes
Appl. {\bf 15} (1983), 155-167.

\bibitem{FS09}
P. Flajolet and R. Sedgewick. {\it Analytic Combinatorics}, Cambr. Univ.
Press, Cambridge, 2009.

\bibitem{F93}
B. Fristedt. The structure of random partitions of large integers.
Trans. Amer. Math. Soc. {\bf 337} (1993), 703-735.

\bibitem{GS92}
W. M. Y. Goh and E. Schmutz. Gap free set partitions. Random Str.
Alg. {\bf 3}(1992), 9-18.

\bibitem{H67}
L. H. Harper. Stirling behavior is asymptotically normal. Ann.
Math. Statist. {\bf 38} (1967), 410-414.

\bibitem{MW55}
L. Moser and M. Wyman. An asymptotic for the Bell numbers. Trans.
Roy. Soc. Canada {\bf49} (1955), 49-55.

\bibitem{M05}
L. Mutafchiev. On the maximal multiplicity of parts in a random
integer partition. Ramanujan J. {\bf 9} (2005), 305-316.


\bibitem{O95}
A. M. Odlyzko. Asymptotic enumeration methods. {\it In: Handbook
of Combinatorics, Vol. II (R. Graham, M. Gr\"{o}tschel and L.
Lov\'{a}sz, Eds.)}. Elsevier, Amsterdam, 1995.

\bibitem{OR85}
A. Odlyzko and L. B. Richmond. On the number of distinct block
sizes in partitions of a set. J. Combin. Theory Ser. A {\bf 38}
(1985), 170-181.


\bibitem{P97}
B. Pittel. Random set partitions: asymptotics of subset counts. J.
Combin. Theory Ser. A 79(1997), 326-359.


\bibitem{S74}
V. N. Sachkov. Random partitions of sets. Theory Probab. Appl.
{\bf 19} (1974), 184-190.

\bibitem{S78}
V. N. Sachkov. {\it Probabilistic Methods in Combinatorial
Analysis (in Russian)}, Nauka, Moscow, 1978; {\it English
translation: Vol.56 of Encyclopedia of Mathematics and Its
Applications}, Cambridge Univ. Press, Cambridge, 1984.


\end{thebibliography}
\end{document}